\documentclass[11pt]{amsart}
\usepackage[margin=3cm]{geometry}

\title{Global estimates on the Brenier Map}

\author{Andrea Bidoia}
\address{Scuola Galileiana di Studi Superiori, Università di Padova; Dipartimento di Matematica "Tullio Levi-Civita", Università di Padova, Via Trieste 63, 35121 Padova, Italy}
\email{andrea.bidoia@studenti.unipd.it}

\usepackage{amsmath,amsfonts,amsbsy,amsgen,amscd,mathrsfs,amssymb,amsthm}
\usepackage{enumerate,mathtools}
\usepackage{cases}

\usepackage{esint}

\usepackage{clipboard}
\newclipboard{myclipboard}

\usepackage[T1]{fontenc} 

\usepackage{url}

\usepackage{mathrsfs}

\def\XXint#1#2#3{{\setbox0=\hbox{$#1{#2#3}{\int}$}
		\vcenter{\hbox{$#2#3$}}\kern-.5\wd0}}

\usepackage[colorlinks=true,linkcolor=blue]{hyperref} 
\usepackage{cleveref}

\makeatletter
\renewcommand*{\eqref}[1]{%
	\hyperref[{#1}]{\textup{\tagform@{\ref*{#1}}}}%
}
\makeatother

\newtheorem{theorem}{Theorem}[section]
\newtheorem*{theorem*}{Theorem}
\newtheorem{lemma}[theorem]{Lemma}

\theoremstyle{definition}

\newtheorem*{remark*}{Remark}
\newtheorem{definition}[theorem]{Definition}
\newtheorem*{definition*}{Definition}

\newtheorem*{acknowledgments}{Acknowledgments}

\AtBeginDocument{
	\setlength{\abovedisplayskip}{4pt}
	\setlength{\belowdisplayskip}{4pt}
	\setlength{\abovedisplayshortskip}{4pt}
	\setlength{\belowdisplayshortskip}{4pt}
}

\newcommand{\R}{\mathbb R}
\newcommand{\Tr}{\operatorname{Tr}}
\newcommand{\Id}{\operatorname{Id}}
\newcommand{\supp}{\mathop{\mathrm{supp}}}
\newcommand{\avgint}{\fint}
\newcommand{\x}{\overline{x}_\varepsilon}
\newcommand{\de}{\partial}

\def\e{\varepsilon}

\numberwithin{equation}{section}

\begin{document}
\newpage
\maketitle

\begin{abstract}
    Caffarelli's contraction theorem and the analogous Laplacian result in \cite{de2024optimal, gozlan2025global} are two examples of how log-Hessian bounds on probability densities yield estimates on the derivative of the corresponding Brenier map with optimal dimensional dependence. The main goal of this paper is to extend such phenomenon to a broader class of convex estimates such as norms.
\end{abstract}

\section{Introduction}
In \cite{caffarelli2000monotonicity}, Caffarelli proved that the Brenier map between a Gaussian measure and a log-concave perturbation is a contraction. A different formulation commonly referred to as "Caffarelli's contraction theorem" is given in \cite{colombo2015lipschitz, kolesnikov2011mass}:
\begin{theorem*}[Caffarelli]
	Let $d\mu=e^{-V}dx,\;d\nu=e^{-W}dx$ be probability measures on $\R^d$ with $V,W\in C^{1,1}_{\rm{loc}}(\R^d)$. Assume that there are constants $\Lambda_V,\lambda_W>0$ such that
	\begin{equation*}
		\nabla^2V\preceq\Lambda_V\Id_d,\quad\nabla^2W\succeq\lambda_W\Id_d\quad\text{a.e. in }\R^d.
	\end{equation*}
	Let $T=\nabla\phi$ be the Brenier map transporting $\mu$ to $\nu$. Then,
	\begin{equation*}
		\|\nabla^2\phi\|_{L^\infty}\leq\sqrt{\Lambda_V/\lambda_W}.
	\end{equation*}
\end{theorem*}
More recently, De Philippis-Shenfeld \cite{de2024optimal} and Gozlan-Sylvestre \cite{gozlan2025global} found an analogous result at the level of the Laplacians, with estimates given on $V$ and $W^*$, the Legendre transform of $W$:

\begin{theorem*}[Gozlan-Sylvestre]
	Let $d\mu=e^{-V}dx,\;d\nu=e^{-W}dx$ be probability measures on $\R^d$ with $W$ convex and $V,W,W^*\in C^{1,1}_{\rm{loc}}(\R^d)$. Assume that there are constants $\Lambda_V,\lambda_W>0$ such that
	\begin{equation*}
		\Delta V\leq\Lambda_V,\quad\Delta W^*\leq 1/\lambda_W\quad\text{a.e. in }\R^d.
	\end{equation*}
	Let $T=\nabla\phi$ be the Brenier map transporting $\mu$ to $\nu$. Then,
	\begin{equation*}
		\|\Delta\phi\|_{L^\infty}\leq\sqrt{\Lambda_V/\lambda_W}.
	\end{equation*}
\end{theorem*}
In \Cref{sec:2}, we use a maximum principle argument on the Monge-Ampère equation to prove Gozlan-Sylvestre's theorem, in a similar fashion to Caffarelli's proof. In \Cref{sec:3}, we see how the above theorems can be generalized for a suitable class of convex functions $\mathbf{G}_d$:

\begin{definition*}
	Denote with $\mathbb{S}_d$ the set of $d\times d$ symmetric real matrices; a continuous function ${f\colon\mathbb{S}_d\to\R}$ is called \emph{good} if it is
	\begin{itemize}
		\item convex
		\item increasing in $\mathbb{S}_d^+$: $0\preceq X\preceq X'\implies f(X)\le f(X'),\,$ and $\,0\precneqq X\implies f(X)>0$
		\item positively homogeneous: $f(tX)=tf(X)\;\forall t>0$
	\end{itemize}
	and we denote with $\textbf{G}_d$ the set of \emph{good} functions.
\end{definition*}

\setcounter{theorem}{0}
\renewcommand{\thetheorem}{\Alph{theorem}}
\begin{theorem}\label{thm_A}
	Let $f\in\mathbf{G}_d$, and $d\mu=e^{-V}dx,\;d\nu=e^{-W}dx$ be probability measures on $\R^d$ with $W$ convex and $V,W,W^*\in C^{1,1}_{\rm{loc}}(\R^d)$. Assume that there are constants $\Lambda_V,\lambda_W>0$ such that
	\begin{equation*}
		f(\nabla^2V)\leq\Lambda_V,\quad f(\nabla^2W^*)\leq 1/\lambda_W\quad\text{a.e. in }\R^d.
	\end{equation*}
	Let $T=\nabla\phi$ be the Brenier map transporting $\mu$ to $\nu$. Then,
	\begin{equation*}\label{good_bound}
		\|f(\nabla^2\phi)\|_{L^\infty}\leq\sqrt{\Lambda_V/\lambda_W}.
	\end{equation*}
\end{theorem}
\renewcommand{\thetheorem}{\thesection.\arabic{theorem}}

The previous bound is sharp as can be seen by taking $$\mu=\mathcal{N}(0, \alpha^2\Id),\quad \nu=\mathcal{N}(0,\beta^2\Id);\quad\Lambda_V=\frac{f(\Id)}{\alpha^2},\quad 1/\lambda_W=\beta^2f(\Id)$$
as in this case $\nabla\phi(x)=\frac{\beta x}\alpha$, which yields $f(\nabla^2\phi)=\frac{\beta}\alpha f(\Id)$.

\begin{remark*}
	By the characterization of convex unitarily invariant functions on $\mathbb{S}_d$ \cite{davis1957all}, if ${g\colon\R^d\to\R}$ is convex and symmetric, then $f(X)\coloneq g(\lambda(X))$ is convex, which lets us build some interesting examples of \emph{good} functions: denote with $\lambda_1\geq\cdots\geq\lambda_d$ the eigenvalues of $X$ and with $\lambda^+$ their positive parts, then
	$$S_k(X)\coloneq \lambda_1+\cdots+\lambda_k,\quad H_k(X)^{1/k}\coloneq\bigg(\prod_{1\leq i_1\leq\cdots\leq i_k\leq d}\lambda_{i_1}^+\cdots\lambda_{i_k}^+\bigg)^{1/k},\quad N_p(X)\coloneq\|\lambda^+\|_p\;\;p\geq1$$
	are \emph{good}, as well as all norms increasing on $\mathbb{S}_d^+$. In particular for $S_1=\lambda_{max}$ and $S_d=\Tr$ we recover Caffarelli's and Gozlan-Sylvestre's theorems respectively. Although most useful estimates involve unitarily invariant functions, we don't require this property.
\end{remark*}

\begin{acknowledgments}
	The author would like to thank Guido De Philippis for the helpful conversations and supervision.
\end{acknowledgments}

\section{The divergence of the Brenier map}\label{sec:2}
The first result in the direction of Laplacians is due to De Philippis and Shenfeld \cite{de2024optimal}, where the same $L^p$ approach of Kolesnikov \cite{kolesnikov2011mass} was used to prove Theorem \ref{lapl} under the stronger assumption $\nabla^2W\succeq d\lambda_W$; this has been weakened by Gozlan and Sylvestre \cite{gozlan2025global} using a different method based on entropic regularisation.
In this section we give a different proof of Theorem \ref{lapl}, following a strategy analogous to Caffarelli's: first we show how the desired estimate follows formally from a maximum principle argument applied to the Monge-Ampère equation, then we replicate the procedure in a rigorous way using incremental quotient approximations to the Laplacian.
\begin{theorem}\label{lapl}
	Let $d\mu=e^{-V}dx,\;d\nu=e^{-W}dx$ be probability measures on $\R^d$ with $W$ convex and $V,W,W^*\in C^{1,1}_{\rm{loc}}(\R^d)$. Assume that there are constants $\Lambda_V,\lambda_W>0$ such that
	\begin{equation*}
		\Delta V\leq\Lambda_V,\quad\Delta W^*\leq 1/\lambda_W\quad\text{a.e. in }\R^d.
	\end{equation*}
	Let $T=\nabla\phi$ be the Brenier map transporting $\mu$ to $\nu$. Then,
	\begin{equation}\label{lapl_target}
		\|\Delta\phi\|_{L^\infty}\leq\sqrt{\Lambda_V/\lambda_W}.
	\end{equation}
\end{theorem}
\begin{proof}[Proof (informal)]
	Assume $V,W,\phi$ are smooth, then $\phi$ satisfies the (log-)Monge-Ampère equation $$V=W(\nabla\phi)-\log\det\nabla^2\phi.$$
	Differentiating two times in direction $e_i$ yields
	\begin{equation}
		\de_{ii}V=\langle\nabla W(\nabla\phi), \nabla \de_{ii}\phi\rangle + \langle \nabla^2W(\nabla\phi)\nabla\de_i\phi, \nabla\de_i\phi\rangle + \Tr\big[(H^{-1}\de_iH)^2\big]-\Tr(H^{-1}\de_{ii}H)
	\end{equation}
	where $H\coloneq\nabla^2\phi$. By summing over the basis $\{e_i\}$ we get
	\begin{equation}\label{laplV}
		\begin{aligned}
			\Delta V&=\langle\nabla W(\nabla\phi), \nabla \Delta\phi\rangle + \Tr(H\nabla^2WH) + \sum\nolimits_i\Tr\big[(H^{-1}\de_iH)^2\big]-\Tr(H^{-1}\nabla^2\Delta\phi)\\
			&\geq\langle\nabla W(\nabla\phi), \nabla \Delta\phi\rangle + \Tr(H\nabla^2W(\nabla\phi)H) - \Tr(H^{-1}\nabla^2\Delta\phi).
		\end{aligned}
	\end{equation}
	where we used $$\Tr\big[(H^{-1}\de_iH)^2\big]=\Tr\big[(H^{-1/2}\de_iHH^{-1/2})^T(H^{-1/2}\de_iHH^{-1/2})\big]\geq0.$$
	Assume further that $\Delta\phi$ attains a maximum at $x_0$, then the optimality conditions give ${\nabla\Delta\phi=0},\,{\nabla^2\Delta\phi\preceq0}$ at $x_0$. Hence \eqref{laplV} implies
	\begin{equation}\label{eq:formal_end}
		\Lambda_V\geq\Delta V\geq\Tr(H\nabla^2W(\nabla\phi)H)\geq\frac{\Tr(H)^2}{\Tr[\nabla^2W(\nabla\phi)^{-1}]}\geq\lambda_W(\Delta\phi)^2\quad\text{at }x_0
	\end{equation}
	where we used $\nabla^2W(x)^{-1}=\nabla^2W^*(y)$ and the Cauchy-Schwarz inequality for the Frobenius scalar product $\langle A,B\rangle\coloneq\Tr(B^TA)$:
	\begin{equation}\label{cs}
		\Tr(HXH)\Tr(X^{-1})=\langle HX^{1/2},HX^{1/2}\rangle\langle X^{-1/2},X^{-1/2}\rangle\geq\langle HX^{1/2},X^{-1/2}\rangle^2=\Tr(H)^2.
	\end{equation} 
	Rearranging \eqref{eq:formal_end} yields $$\|\Delta\phi\|_{L^\infty}=\Delta\phi(x_0)\leq\sqrt{\Lambda_V/\lambda_W}$$ which concludes the proof.
\end{proof}
While the regularity assumptions can be met by approximation, the main obstacle to this formal argument is the existence of a maximum for $\Delta\phi$. The following lemma due to Caffarelli (a short proof of which we include in the Appendix) motivates the use of $\e$-approximations to the Laplacian, an analogous tool to the incremental quotients originally used in his proof.
\begin{lemma}\label{finite_max}
	Let $d\mu = f(x)\,dx,\,d\nu = g(x)\,dx \in \mathcal{P}(\R^{d})$ with finite second moments and $\nabla \phi = T$ be the optimal transport taking $\mu$ to $\nu$.
	Assume that $\log f \in L^\infty_{\rm loc}(\R^d)$ and that $g$ is bounded away from zero in the ball $B_{j}$ for some $j > 0$ and vanishes outside $B_{j}$. 
	Then,
	$$
	T(x) \to j \frac{x}{|x|}\qquad \text{uniformly as }|x| \to\infty.
	$$
	In particular, for any fixed $\e > 0$, the function $\phi(x+\e y)+\phi(x-\e y)-2\phi(x) \to 0$ as $|x| \to \infty$ uniformly for $y\in\mathbb{S}^{d-1}$.
\end{lemma}

We mention a regularity result due to Cordero-Erausquin and Figalli \cite{reg_figalli}:

\begin{theorem}\label{thm:fig-reg}
	Let $Y \subset \R^{d}$ be a convex open set, and $f : \R^d\to \R^{+}$ and $g : Y \to \R^{+}$ be probability densities locally bounded away from zero and infinity. Then for any set $X'\subset \subset \R^d$, the optimal transport map $T = \nabla\phi\colon\R^d \to Y$ between $f(x)\, dx$ and $g(y)\, dy$ is of class $C^{0,\alpha}(X')$ for some $\alpha > 0$.
	In addition, if $f \in C^{k,\beta}_{\rm{loc}}(\R^d)$ and $g \in C^{k,\beta}_{\rm{loc}}(Y)$ for some $k\ge0$ and $\beta \in (0,1)$, then $\phi \in C^{k+2,\beta}_{\rm{loc}}(\R^d)$.
\end{theorem}

We will also recall the stability of optimal transport maps under approximation, as stated in \cite{colombo2015lipschitz}: let $\{f_j\}$ and $\{g_j\}$ be locally uniformly bounded probability densities such that $f_{j} \to f$ and $g_{j} \to g$ in $L^{1}_{\rm{loc}}$. Then, the associated potentials $\phi_{j} \to \phi$ locally uniformly and $\nabla\phi_{j} \to \nabla\phi$ in measure. This stability result is enough for our purpose, in particular:
\begin{equation*}
	\|\Delta\phi\|_{L^\infty}=\sup_{\substack{\psi\in C^\infty_c\\ \|\psi\|_{L^1}=1}}\int_{\R^d} \phi\Delta\psi dx=\sup_{\substack{\psi\in C^\infty_c\\ \|\psi\|_{L^1}=1}}\lim_{j\to\infty}\int_{\R^d} \phi_j\Delta\psi dx\leq\liminf_{j\to\infty}\|\Delta\phi_j\|_{L^\infty}.
\end{equation*}

\begin{lemma}\label{lemma_approx}
	By approximation, we may assume that $V$ is smooth and $V\geq C_1+C_2|x|^2$ for some $C_2>0$, that $W$ is equal to infinity outside the ball $\overline{B_R}$ and that $W$ is smooth in $\overline{B_R}$ for some $R>0$.
\end{lemma}
\begin{proof}
	Let $\eta$ be a standard convolution kernel (radially symmetric, non-negative, compactly supported), and define $e^{-V_t}\coloneq e^{-V}\ast\eta_t$, then we obtain $$	V_t=-\log(e^{-V}\ast\eta_t),\quad\nabla V_t=\mathbb{E}[\nabla V]\quad\text{where }\quad\mathbb{E}[X](x)\coloneq\frac{(Xe^{-V})\ast\eta_t}{e^{-V}\ast\eta_t}(x),$$
	and by {Cauchy-Schwarz} inequality: $$\nabla^2 V_t=\mathbb{E}[\nabla^2 V]-\mathbb{E}[\nabla V\otimes\nabla V]+\mathbb{E}[\nabla V]\otimes\mathbb{E}[\nabla V]\preceq\mathbb{E}[\nabla^2 V]$$ which implies $\Delta V_t\leq\mathbb{E}[\Delta V]\leq\Lambda_V$. Taking $\supp\eta_t\subset B_{\delta_t}$ yields $$e^{-V_t(x)}\lesssim\int_{B(x,\delta_t)}e^{-V}dy\to0\text{ as }|x|\to\infty,$$ and in particular, $V_t$ is positive outside a bounded set.
	Clearly $(e^{-V_t})_{t\to0}$ satisfies the necessary stability conditions, therefore we may assume $V$ is smooth and positive outside a bounded set. Up to adjusting $V_t\coloneq (1-t)V+\frac{t\Lambda_V}d\frac{|x|^2}2+c_t$ we may also assume $V\geq C_1+C_2|x|^2$ for some $C_2>0$. More precisely,
	$$\Delta V_t=(1-t)\Delta V+t\Lambda_V\leq\Lambda_V,\quad e^{-(1-t)V-\frac{t\Lambda_V}d\frac{|x|^2}2}\to e^{-V}\text{ locally uniformly},$$
	and by dominated convergence, $$e^{c_t}=\int_{\R^d} e^{-(1-t)V-\frac{t\Lambda_V}d\frac{|x|^2}2}dx\to1$$
	where we used $(1-t)V+\frac{t\Lambda_V}d\frac{|x|^2}2\geq\min\{V, \frac{\Lambda_V}d\frac{|x|^2}2\}$, which together imply $e^{-V_t}\to e^{-V}$ locally uniformly.
	Next if we define $$W_t(x)\coloneq\begin{cases}
		(W\ast\eta_t)(x)+c_t & \text{ in } \overline{B_{1/t}}\\
		+\infty & \text{ in } \R^d\setminus \overline{B_{1/t}},
	\end{cases}$$
	we obtain, by Jensen inequality and the convexity of the map $\mathbb{S}_d^+\ni A\mapsto A^{-1}$ \cite[V.1.15]{bhatia2013matrix}
	$$\nabla^2W_t^*(y)=(\nabla^2W\ast\eta_t)(x)^{-1}\preceq((\nabla^2W)^{-1}\ast\eta_t)(x) \text{ for } x\in\partial W_t^*(y)$$ which yields $\Delta W_t^*\leq \Delta W^*\ast\eta_t\leq 1/\lambda_W$. The estimate $\Delta W^*\leq 1/\lambda_W$ also implies ${\nabla^2W\succeq \lambda_W\Id}$, therefore 
	$$W(x)\geq W(0)+\langle\nabla W(0),x\rangle+\frac{\lambda_W}2|x|^2\geq C_1'+C_2'|x|^2$$ for some $C_2'>0$. Similarly, for some $C_2''>0$ we have
	$$W\ast\eta_t(x)\geq\inf_{B(x,\delta_t)}W\geq C_1'+C_2'|x-\delta_t|^2\geq C_1''+C_2''|x|^2.$$
	Lastly, $e^{-W\ast\eta_t}\to e^{-W}\text{ locally uniformly}$, and by dominated convergence we have $$e^{c_t}=\int_{B_{1/t}}e^{-W\ast\eta_t}dx\to1,$$
	which together imply $e^{-W_t}\to e^{-W}\text{ locally uniformly}.$ 
\end{proof}

\begin{definition}
	Let $y\in\R^d,\,\e>0$, define the increments:
	\begin{gather*}
		(\delta_yf)(x)\coloneq f(x+y)-f(x),\quad(\delta^2_yf)(x)\coloneq f(x+y)+f(x-y)-2f(x),\\
		(\Delta_\e f)(x)\coloneq\avgint_{\mathbb{S}^{d-1}}(\delta_{\e y}f)(x)dy=\frac1{2}\avgint_{\mathbb{S}^{d-1}}(\delta^2_{\e y}f)(x)dy.
	\end{gather*}
\end{definition}
We will leave in the Appendix the proof of some properties of $\Delta_\e$, as shown in \cite{de2024optimal}:
\begin{lemma}\label{lapl_lemma}
	Let $\e>0$ and $f\colon\R^d\to\R$, denote with $\Delta f$ its distributional Laplacian. Then,
	$$\lim_{\e\to0}\frac{\Delta_\e f}{\e^2}=\frac{\Delta f}{2d}\text{ in the sense of distributions}.$$
	Further, if $\Delta f\leq\ell$, then $$\Delta_\e f\leq\frac{\ell}{d}\frac{\e^2}2,\quad\forall\e>0.$$
\end{lemma}
With this tool we are ready for the proof of Theorem \ref{lapl}:
\begin{proof}
	By Lemma \ref{lemma_approx} and regularity of the optimal transport, the map $T=\nabla\phi$ is smooth and satisfies the (log-)Monge-Ampère equation $$V=W(\nabla\phi)-\log\det H.$$
	Taking the increments $2\Delta_\e$ on both sides yields 
	\begin{equation}\label{delta_monge}
		\frac{\e^2}d\Lambda_V\geq2\Delta_\e V=2\Delta_\e(W(\nabla\phi))-2\Delta_\e(\log\det H).
	\end{equation}
	By Lemma \ref{finite_max}, $0\leq\Delta_\e\phi(x)\to0$ uniformly as $|x|\to\infty$, therefore $\Delta_\e\phi$ attains a global maximum at some $\x\in\R^d$ where the optimality conditions read $\Delta_\e\nabla\phi(\x)=0$, $\Delta_\e H(\x)\preceq0$.
	We estimate each term in the right side: the concavity of $\log\det$ implies
	\begin{equation}\label{log_concav}
		\begin{gathered}
			\delta_{\e y}(\log\det H)(\x)\leq\langle H(\x)^{-1}, \delta_{\e y}H(\x)\rangle,\\
			\Delta_\e(\log\det H)(\x)\leq\langle H(\x)^{-1},\Delta_\e H(\x)\rangle\leq 0.
		\end{gathered}
	\end{equation}
	The following lemma deals with the term $W(\nabla\phi)$:
	\begin{lemma}
		With the previous notation, $\forall\eta>0,\,\exists\overline{\e}>0$ such that for all $0<\e\leq\overline{\e}$, the following estimate holds:
		$$2\Delta_\e (W(\nabla\phi))(\x)\geq(1-\eta)\avgint_{\mathbb{S}^{d-1}}\langle\delta_{\e y}\nabla\phi(\x),\nabla^2W(\nabla\phi(\x))\delta_{\e y}\nabla\phi(\x)\rangle dy.$$
	\end{lemma}
	\begin{proof}
		We can write
		\begin{equation}\label{remainder}
			2\delta_{\e y}(W(\nabla\phi))(\x)=2\langle\nabla W(\nabla\phi(\x)),\delta_{\e 		y}\nabla\phi(\x)\rangle+\langle\delta_{\e y}\nabla\phi(\x),X\delta_{\e y}\nabla\phi(\x)\rangle+R_{\e}(y)
		\end{equation}
		where $X\coloneq\nabla^2W(\nabla\phi(\x))$ and estimate the remainder $|R_\e(y)|\leq K|\delta_{\e y}\nabla\phi(\x)|^3$ as $\nabla^3W$ is bounded. Now if our point of maximum $\x$ stays in a bounded ball as $\e\to0$, then $|\delta_{\e y}\nabla\phi(\x)|{\lesssim\e}$.
		If instead (up to subsequence) $|\x|\to\infty$, then by Lemma \ref{finite_max} ${|\delta_{\e y}\nabla\phi(\x)|\to0}$ uniformly for $y\in\mathbb{S}^{d-1}$. Either way $\forall\eta>0$, for $\e$ small enough we obtain $$|R_\e(y)|\leq K|\delta_{\e y}\nabla\phi(\x)|^3\leq K\eta|\delta_{\e y}\nabla\phi(\x)|^2\leq\frac{K\eta}{\lambda_W}\langle\delta_{\e y}\nabla\phi(\x),X\delta_{\e y}\nabla\phi(\x)\rangle.$$
		Up to renormalizing $\eta$ we can integrate \eqref{remainder} to obtain the desired estimate
		\begin{equation}\label{eq:deltaW}
			2\Delta_\e (W(\nabla\phi))(\x)\geq(1-\eta)\avgint_{\mathbb{S}^{d-1}}\langle\delta_{\e y}\nabla\phi(\x),X\delta_{\e y}\nabla\phi(\x)\rangle dy.
		\end{equation}
	\end{proof}
	Combining \eqref{delta_monge}, \eqref{log_concav} and \eqref{eq:deltaW} yields
	\begin{equation}\label{ineq_1}
		\frac{\e^2}d\frac{\Lambda_V}{1-\eta}\geq\avgint_{\mathbb{S}^{d-1}}\langle\delta_{\e y}\nabla\phi(\x)X\delta_{\e y}\nabla\phi(\x)\rangle dy.
	\end{equation}
	Now define $F(y)\coloneq X^{1/2}\delta_{\e y}\nabla\phi(\x)$ and $G(y)\coloneq X^{-1/2}y$, then we have the following {Cauchy-Schwarz} inequality analogous to \eqref{cs}:
	\begin{equation}\label{discrete_cs}
		\begin{aligned}
			\bigg(\avgint_{\mathbb{S}^{d-1}}\langle\delta_{\e y}\nabla\phi(\x),y\rangle dy\bigg)^2&=	\bigg(\avgint_{\mathbb{S}^{d-1}}\langle F(y),G(y)\rangle dy\bigg)^2\leq\avgint_{\mathbb{S}^{d-1}}|F|^2 dy\avgint_{\mathbb{S}^{d-1}}|G|^2 dy\\
			&=\avgint_{\mathbb{S}^{d-1}}\langle \delta_{\e y}\nabla\phi(\x),X\delta_{\e y}\nabla\phi(\x)\rangle dy\avgint_{\mathbb{S}^{d-1}}\langle y,X^{-1}y\rangle dy.
		\end{aligned}
	\end{equation}
	By combining \eqref{ineq_1}, \eqref{discrete_cs} and using $$\avgint_{\mathbb{S}^{d-1}}\langle y,X^{-1}y\rangle dy=\frac{\Tr(X^{-1})}{d}\leq\frac{1}{d\lambda_W}$$
	we obtain the estimate:
	\begin{equation}\label{ineq_2}
		\avgint_{\mathbb{S}^{d-1}}\langle\delta_{\e y}\nabla\phi(\x),y\rangle dy\leq\frac{\e}{d}\frac{C}{\sqrt{1-\eta}}\quad\text{where }C\coloneq\sqrt{\Lambda_V/\lambda_W}.
	\end{equation}
	To relate this quantity to $\Delta_\e\phi(\x)$, notice that the convexity of $\phi$ yields
	$$\phi(\x)\geq\phi(\x\pm\e y)\mp\e\langle\nabla\phi(\x\pm\e y),y\rangle$$ and $$\delta^2_{\e y}\phi(\x)\leq\e\langle\nabla\phi(\x+\e y)-\nabla\phi(\x-\e y),y\rangle,$$
	which integrated gives
	\begin{equation}\label{ineq_3}
		\Delta_{\e}\phi(\x)\leq\frac{\e}2\avgint_{\mathbb{S}^{d-1}}\langle\nabla\phi(\x+\e y)-\nabla\phi(\x-\e y),y\rangle dy=\e\avgint_{\mathbb{S}^{d-1}}\langle\delta_{\e y}\nabla\phi(\x),y\rangle dy.
	\end{equation}
	In particular, from \eqref{ineq_2} and \eqref{ineq_3} we have
	\begin{equation}\label{2C}
		\Delta_\e\phi(\x)\leq\frac{\e^2}{2d}\frac{2C}{\sqrt{1-\eta}}\implies\|\Delta\phi\|_{L^\infty}\leq\frac{2C}{\sqrt{1-\eta}}.
	\end{equation}
	This first estimate is the desired one up to a factor of 2; to remove this factor, we use a bootstrap argument similar to Caffarelli's. Suppose we have obtained $\Delta\phi\leq\frac{aC}{\sqrt{1-\eta}}$ for some $1<a\leq2$, then
	$$\delta^2_{\e y}\phi(\x)=\int_0^\e\langle\nabla\phi(\x+ty)-\nabla\phi(\x-ty),y\rangle dt\implies\Delta_{\e}\phi(\x)=\avgint_{\mathbb{S}^{d-1}}\int_0^\e\langle\delta_{ty}\nabla\phi(\x),y\rangle dtdy.$$
	By the convexity of $\phi$, we have $\langle\delta_{ty}\nabla\phi(\x),y\rangle\leq\langle\delta_{\e y}\nabla\phi(\x),y\rangle$ for all $0\leq t\leq\e$ thus
	$$\avgint_{\mathbb{S}^{d-1}}\langle\delta_{ty}\nabla\phi(\x),y\rangle dy\leq\avgint_{\mathbb{S}^{d-1}}\langle\delta_{\e y}\nabla\phi(\x),y\rangle dy\leq\frac{\e}{d}\frac{C}{\sqrt{1-\eta}}\quad\text{by }\eqref{ineq_2}$$
	but also $$\avgint_{\mathbb{S}^{d-1}}\langle\delta_{ty}\nabla\phi(\x),y\rangle dy=\frac{t}{d}\avgint_{B_1}\Delta\phi(\x+tz)dz\leq\frac{t}{d}\frac{aC}{\sqrt{1-\eta}}.$$
	Using the first upper bound for $t\geq\e/a$ and the second for $t<\e/a$ we derive
	\begin{equation*}
		\Delta_\e\phi(\x)\leq\frac{C}{d\sqrt{1-\eta}}\bigg[\int_0^{\e/a}at dt+\int_{\e/a}^\e \e dt\bigg]=\frac{\e^2}{2d}\frac{C(2-1/a)}{\sqrt{1-\eta}}
	\end{equation*}
	and by Lemma \ref{lapl_lemma}, $$\|\Delta\phi\|_{L^\infty}\leq\frac{C(2-1/a)}{\sqrt{1-\eta}}.$$
	Starting from $a_0\coloneq2$ and repeating the same procedure with $a_{n+1}\coloneq 2-1/{a_n}$ an infinite number of times, then letting $\eta\to0$ we conclude the proof of \eqref{lapl_target} as $a_n$ decreases to 1.
\end{proof}
\section{Proof of Theorem A}\label{sec:3}
Both Caffarelli's and Gozlan-Sylvestre's theorems are examples of how upper bounds for $f(\nabla^2V)$ and $f(\nabla^2W^*)$ provide upper bounds for $f(\nabla^2\phi)$, for some estimating function $f$ equal to $\lambda_{max},\Tr$ respectively. In this section we use elementary convex analysis to generalize this result to a broader class of functions.
Denote with $\mathbb{S}_d$ the set of $d\times d$ symmetric real matrices, with $\mathbb{S}_d^+$ the positive semidefinite ones, and with $\mathbb{S}_d^{++}$ the positive definite ones.
\begin{lemma}\label{anis}
	Let $Y\in\mathbb{S}_d^{++}$, then the statement of Theorem \ref{thm_A} holds for $f(X)=\langle X,Y\rangle$.
\end{lemma}
\begin{proof}
	Let $T=\nabla\phi$ be the Brenier map transporting $\mu$ to $\nu$, and $A\coloneq Y^{1/2}$. Define the following: 
	\begin{gather*}
		\phi_A(x)\coloneq\phi(Ax),\\
		V_A(x)\coloneq V(Ax)-\log\det A,\\
		W_A(x)\coloneq W(A^{-1}x)+\log\det A.
	\end{gather*}
	Then the probability measures $d\mu_A\coloneq e^{-V_A}dx,\,d\nu_A\coloneq e^{-W_A}dx$ satisfy the Laplacian estimates: 
	$$\Delta V_A(x)=\langle\nabla^2 V(Ax), Y\rangle\leq\Lambda_V,\quad\Delta W_A^*(y)=\langle\nabla^2 W^*(Ay), Y\rangle\leq 1/\lambda_W\quad\text{a.e. in }\R^d.$$
	We can easily verify that $\phi_A$ is convex and $(\nabla\phi_A)_\#\mu_A=\nu_A$, therefore $\nabla\phi_A$ is the Brenier map and Theorem \ref{lapl} yields $$\|\langle\nabla^2\phi, Y\rangle\|_{L^\infty}=\|\Delta\phi_A\|_{L^\infty}\leq\sqrt{\Lambda_V/\lambda_W}.$$
\end{proof}


\begin{proof}[Proof of Theorem \ref{thm_A}]
	Denote by $$\de f(X)\coloneq\{Y\in\mathbb{S}_d\colon f(Z)\ge f(X)+\langle Z-X,Y\rangle\;\;\forall Z\in\mathbb{S}_d\}$$ the subdifferential of $f$ at $X$. By convexity and homogeneity \cite[Theorem 13.2]{rockafellar1997convex}, we can write
	$$f(X)=\sup_{Y\in\de f(0)}\langle X,Y\rangle\;\text{ with equality for }Y\in\de f(X)\subset\de f(0).$$
	Observe that for any $X\in\mathbb{S}_d^{++}$, the subdifferential satisfies $\de f(X)\subset\mathbb{S}_d^+$. Indeed, for any $Y\in\partial f(X)$ and $v\in\R^d$, choose $\alpha>0$ such that $0\preceq X-\alpha vv^T\preceq X$. Then,
	$$\langle X-\alpha vv^T,Y\rangle\leq f(X-\alpha vv^T)\le f(X)=\langle X,Y\rangle,$$
	which implies $\langle v,Yv\rangle\ge0$. By the monotonicity of $f$, $$\beta\coloneq\min\{f(X)\colon X\succeq0,\,\Tr(X)=1\}>0,$$ hence for any $X\succeq0$, the estimate $f(X)\ge\beta\Tr(X)$ holds. Our next goal is to obtain a positive definite matrix $E\in\de f(0)$, and we will do so in two steps: \begin{enumerate}
		\item first we prove that there is no non-zero vector $v$ in the intersection of $\ker Y$ for all ${0\preceq Y\in\de f(0)}$;
		\item second, we iteratively construct a positive definite matrix $E\in\de f(0)$.
	\end{enumerate}
	
	\emph{Step 1.} Suppose that there exists a vector $v\neq0$ such that $$Yv=0\quad\text{for all}\quad0\preceq Y\in\de f(0).$$ Then by writing $$f(X)=\sup_{0\preceq Y\in\de f(0)}\langle X,Y\rangle\quad\text{for any}\;X\succ0,$$ we obtain $$f(X)=f(X+\alpha vv^T)\ge\beta\Tr(X) + \alpha\beta|v|^2\quad\forall\alpha>0$$ which yields a contradiction as $\alpha\to+\infty$.
	
	\emph{Step 2.} Let $v\neq0$ and denote with $Y_v$ a matrix satisfying $Y_vv\neq0$ and $0\preceq Y_v\in\de f(0)$. If $Y_1\coloneq Y_v$ is non-singular, we choose $E\coloneq Y_1$. Otherwise, let $0\neq v_2\in\ker Y_1$, and define
	$$0\preceq Y_2\coloneq\frac{Y_1+Y_{v_2}}2\in\de f(0).$$ Notice that if $Y_2w=0$ for some vector $w$, then $$\langle w,Y_1w\rangle + \langle w,Y_{v_2}w\rangle=|Y_1^{1/2}w|^2+|Y_{v_2}^{1/2}w|^2=0,$$ which implies $Y_1w=0$ and $Y_{v_2}w=0$. Hence $\ker Y_2\lneqq\ker Y_1$, and repeating the same construction on $Y_2$ yields a positive definite matrix $E\in\de f(0)$ after a finite number of steps.\\
	
	Let $0\preceq Y\in\de f(0)$ and $\e>0$, and notice that 
	\begin{gather*}
		\langle \nabla^2V, Y+\e E\rangle\leq(1+\e)f(\nabla^2V)\leq(1+\e)\Lambda_V,\\
		\langle \nabla^2W^*, Y+\e E\rangle\leq(1+\e)f(\nabla^2W^*)\leq(1+\e)/\lambda_W\quad\text{a.e. in}\,\R^d.
	\end{gather*}
	Hence, applying Lemma \ref{anis} to $Y+\e E\in\mathbb{S}_d^{++}$ yields $$\|\langle\nabla^2\phi, Y\rangle\|_{L^\infty}\leq(1+\e)\sqrt{\Lambda_V/\lambda_W}.$$ Taking $Y$ in a countable dense subset of $\de f(0)\cap\mathbb{S}_d^+$, we obtain $$\langle\nabla^2\phi(x), Y\rangle\leq\sqrt{\Lambda_V/\lambda_W}\quad\text{for all}\;Y\in\de f(0)\cap\mathbb{S}_d^+,\; x-a.e.$$ 
	In particular, up to a set of Lebesgue measure zero we have
	\begin{align*}
		f(\nabla^2\phi(x))&\leq f(\nabla^2\phi(x)+\e\Id)=\langle \nabla^2\phi(x)+\e\Id, Y_\e\rangle\\
		&=\langle \nabla^2\phi(x),Y_\e\rangle + \langle \e\Id,Y_\e\rangle\leq\sqrt{\Lambda_V/\lambda_W}+\e f(\Id)
		\end{align*}
	where $\e>0$ and $0\preceq Y_\e\in\de f(\nabla^2\phi(x)+\e\Id)$. Taking a countable sequence of $\e\to0^+$ concludes the proof.
\end{proof}

\appendix
\section{}
In this appendix we prove two technical results, namely Caffarelli's lemma on the decay of the incremental quotient $\delta^2_{\e y}\phi$, and the approximation properties of the operator $\Delta_\e$.

\begin{lemma}
	\label{lemma:finite max (Gen)}
	Let $d\mu=f(x)\,dx,\, d\nu=g(x)\,dx\in\mathcal{P}(\R^d)$ with finite second moments and $\nabla\phi=T$ be the optimal transport taking $\mu$ to $\nu$.
	Assume that $\log f\in L^\infty_{\rm loc}(\R^d)$ and that $g$ is bounded away from zero in the ball $B_{j}$ for some $j>0$ and vanishes outside $B_{j}$. 
	Then,
	$$T(x)\to j\frac{x}{|x|}\qquad\text{uniformly as }|x|\to\infty.$$
	In particular, for any fixed $\e>0$, the function $\phi(x+\e y)+\phi(x-\e y)-2\phi(x)\to0$ as $|x|\to\infty$, uniformly for $y\in\mathbb{S}^{d-1}$.
\end{lemma}

\begin{proof}
	We begin by noticing that, as a consequence of Theorem \ref{thm:fig-reg}, $T$ is continuous on $\R^d$ and, in particular, the map $T$ is well defined at every point.
	
	Let $x_{0} \in \R^d$ and $t \in (0, \pi/4)$ be fixed, and consider the cone with vertex at $T(x_{0})$ and pointing in the $x_{0}$-direction
	$$\Gamma\coloneq \bigg\{ y \in \R^d\colon \angle (x_{0}, y-T(x_{0})) \leq \frac{\pi}{2} - t \bigg\}.$$
	By the convexity of $\phi$ we see that
	$$\angle (x-x_{0}, T(x)-T(x_{0})) \leq \frac{\pi}2,$$
	hence
	$$\angle (x-x_0,x_0) \leq \angle (x-x_{0}, T(x)-T(x_{0})) + \angle (x_{0}, T(x)-T(x_{0}))
	\leq \pi - t  \qquad \forall \,x \text{ s.t. } \,T(x) \in \Gamma,$$
	and so, up to a set of measure zero, the preimage of $\Gamma$ under $T$ is contained in the (concave) cone
	$$\Omega\coloneq \{ x \in \R^{n}\colon \angle(x_{0}, x-x_{0}) \leq \pi - t \}. $$
	Moreover, since $T_{\#}\mu = \nu$, 
	$$\inf_{x \in B_{j}}g(x)\, \mathcal{L}^{n}(\Gamma \cap B_{j}) \leq \nu(\Gamma \cap B_{j}) = \nu(\Gamma) \leq \mu(\Omega).$$
	Let $B = B_{(|x_{0}|\tan{t})/2}$, and notice that $\Omega \subseteq \R^d \setminus B$. This proves that $\mu(\Omega) \leq \mu(\R^d \setminus B)$.
	
	Now, $\mu(\R^d \setminus B) \to 0$ as $|x_{0}| \to \infty$ since $B$ covers $\R^d$ as $|x_{0}| \to \infty$. Recalling that $g$ is bounded away from zero in $B_{j}$, we have that
	$$ \lim_{|x_{0}| \to \infty} \mathcal{L}^d(\Gamma \cap B_{j}) = 0. $$
	Letting $t \to 0$, we see that $T(x_{0}) \to j \frac{x_{0}}{|x_{0}|}$. As the point $x_{0}$ was fixed arbitrarily, $${\nabla \phi(x) = T(x) \to j \frac{x}{|x|}}\text{ uniformly as }|x| \to \infty.$$ Thus, for any fixed $\e > 0$,
	\begin{gather*}
		|\phi(x+\e y)+\phi(x-\e y)-2\phi(x)|=\e|\langle\nabla\phi(x+\eta y)-\nabla\phi(x-\eta y),y\rangle|\quad\text{ for some }\eta\in[0,\e]\\
		\leq\e j\bigg[\bigg|\frac{x+\eta y}{|x+\eta y|}-\frac{x-\eta y}{|x-\eta y|}\bigg|+o(1)\bigg] \to 0\text{ as }|x| \to \infty, \text{ uniformly for }y\in\mathbb{S}^{d-1}.
	\end{gather*}
\end{proof}

\begin{lemma}
	Let $\e>0$ and $f\colon\R^d\to\R$, denote with $\Delta f$ its distributional Laplacian. Then,
	$$\lim_{\e\to0}\frac{\Delta_\e f}{\e^2}=\frac{\Delta f}{2d}\text{ in the sense of distributions}.$$
	Further, if $\Delta f\leq\ell$, then
	\begin{equation}\label{lapl_e}
		\Delta_\e f\leq\frac{\ell}{d}\frac{\e^2}2,\quad\forall\e>0.
	\end{equation}
\end{lemma}
\begin{proof}
	First note that letting $ \nabla^2f$ be the distributional derivative of $f$ we have
	$$\lim_{\e\to 0}\frac{\delta^2_{\e y}f}{\e^2}=\langle \nabla^2f\,y,y\rangle,$$
	in the sense of distributions, where the distribution $\langle \nabla^2f\,y, y \rangle$ is defined as
	$$\left \langle \langle \nabla^2f\,y, y \rangle, \eta \right \rangle =\int f(x) \langle \nabla^2\eta(x) y, y \rangle  dx,$$
	for test functions $\eta$. Hence, $$\lim_{\e\to 0}\frac{\Delta_{\e}f}{\e^2}=\frac{1}{2}\lim_{\e\to 0}\bigg[\avgint_{\mathbb{S}^{d-1}}\frac{\delta^2_{\e y}f}{\e^2}dy\bigg]=\frac{1}{2}\avgint_{\mathbb{S}^{d-1}}\langle \nabla^2f\,y,y\rangle dy.$$
	For each $i,j\in\{1,\dots,d\}$, by the symmetry of $\mathbb{S}^{d-1}$ under $y_j\mapsto -y_j$,
	$$\avgint_{\mathbb{S}^{d-1}}y_iy_j dy=\delta_{ij}\avgint_{\mathbb{S}^{d-1}}y_i^2 dy=\delta_{ij}\frac{1}{d}.$$
	It follows that
	\begin{align*}
		\lim_{\e\to 0}\frac{\Delta_{\e}f}{\e^2}=\frac{1}{2}\sum_{i=1}^{d}\partial_{ii}^2f\avgint_{\mathbb{S}^{d-1}}y_i^2 dy=\frac{1}{2d}\sum_{i=1}^{d}\partial_{ii}^2f=\frac{\Delta f}{2d}.
	\end{align*}
	Next we prove \eqref{lapl_e}: we have
	$$\Delta_{\e}f(x)=\int_0^{\e} \frac{d}{dr}\bigg[\avgint_{\mathbb{S}^{d-1}}\big(f(x+ry)-f(x)\big)dy\bigg]dr.$$
	Integration by parts yields
	\begin{align*}
		&\frac{d}{dr}\bigg[\avgint_{\mathbb{S}^{d-1}}\big(f(x+ry)-f(x)\big)dy\bigg]=\avgint_{\mathbb{S}^{d-1}}\frac{\partial}{\partial r}\big(f(x+ry)-f(x)\big)dy=\avgint_{\mathbb{S}^{d-1}}\langle\nabla f(x+ry),y\rangle dy\\
		&=\avgint_{\mathbb{S}^{d-1}}\frac{1}{r}\langle\nabla_y [f(x+ry)],y\rangle dy=\frac{1}{d}\avgint_{ B_1}\frac{1}{r}\Delta_y[f(x+ry)]dy=\frac{r}{d}\avgint_{B_1}\Delta f(x+ry)dy\leq\frac{\ell r}{d},
	\end{align*}
	where the last inequality used $\Delta f\le \ell$. It follows that $$\Delta_{\e}f(x)\leq\frac{\ell}{d}\int_0^{\e}r dr=\frac{\ell}{d}\frac{\e^2}{2}.$$
\end{proof}

\bibliographystyle{plain}
\nocite{*}
\bibliography{ref.bib}

\end{document}